\documentclass{amsart}

\usepackage[T1]{fontenc}
\usepackage[utf8]{inputenc}
\usepackage{amsmath,amsthm,amsfonts,amscd,amssymb,eucal,latexsym,mathrsfs}
\usepackage{stmaryrd}
\usepackage{enumerate}
\usepackage{hyperref}
\usepackage[all]{xy}
\usepackage{etoolbox}

\usepackage{tikz,tikz-cd}
\usetikzlibrary{shapes.geometric}
\usetikzlibrary{arrows}
\usepackage{tqft}

\usetikzlibrary{positioning}
\usepackage{float}
\usepackage{MnSymbol}
\usetikzlibrary{matrix}

\newcommand{\toc}{\tableofcontents}

\theoremstyle{plain}
\newtheorem{theorem}{Theorem}[section]
\newtheorem*{theorem*}{Theorem}

\newtheorem*{corollary*}{Corollary}

\newtheorem{lemma}[theorem]{Lemma}

\theoremstyle{definition}
\newtheorem{remark}[theorem]{Remark}

\newtheorem{definition}[theorem]{Definition}
\newtheorem*{definition*}{Definition}

\usetikzlibrary{calc,graphs}
\usepackage{xcolor}
\usetikzlibrary{arrows,decorations.pathmorphing}

\DeclareMathOperator{\Min}{\mathrm{Min}}

\DeclareMathOperator{\Lk}{\mathrm{Lk}}

\newcommand{\acts}{\curvearrowright}

\newcommand{\e}{\varepsilon}

\newcommand{\IR}{\mathbb{R}}

\newcommand{\ZI}{\mathbb{Z}}

\DeclareMathOperator{\Aut}{\mathrm{Aut}}

\DeclareMathOperator{\lra}{\leftrightarrow}

\DeclareMathOperator{\surj}{\twoheadrightarrow}

\newcommand{\SL}{\mathrm{SL}}

\newcommand{\ts}{\textsection}

\newcommand{\ip}[1]{\langle#1\rangle} 

\author{Sylvain Barr\'e}
\author{Mika\"el Pichot}
\address{Sylvain Barr\'e, UMR 6205, LMBA, Université de Bretagne-Sud,BP 573, 56017, Vannes, France}\email{Sylvain.Barre@univ-ubs.fr}
\address{Mika\"el Pichot, McGill University, 805 Sherbrooke St W., Montr\'eal, QC H3A 0B9, Canada}\email{pichot@math.mcgill.ca}

\title{A  virtual geometric action of a braid group}
\begin{document}

\begin{abstract} 
We study a geometric action on a CAT(0) space of a finite index subgroup of the quotient group of the braid group $B_4$ on 4 strands by its center. 
\end{abstract}

\maketitle

\section{Introduction}

Let $G$ be a countable group and $X$ be a CAT(0) space. 
We say that $G$ \emph{acts  virtually geometrically} on $X$ if there exists a group $H$ of finite index in $G$ which admits a geometric (i.e., isometric, properly discontinuous, and cocompact) action on $X$. 

Let $B_4$ be the braid on 4 strings
\[
B_4=\ip{a,b,c\mid aba=bab,\ bcb=cbc,\ ac=ca};
\]
let $Z$ be the center of $B_4$ (it is isomorphic to $\ZI$); the quotient group $B_4/Z$ acts geometrically, by the work of  Brady  \cite{brady1994automatic,brady2000artin},   on a CAT(0) space of dimension 2. This space, henceforth denoted  $X_0$, is a simplicial complex known as the Brady complex, and the corresponding action is called the standard action of $B_4/Z$.

\begin{theorem}\label{T - virt geom}
There exists a CAT(0) complex $X_1$ of dimension 2 with the following two properties
\begin{enumerate}
\item $B_4/Z$ acts virtually geometrically on $X_1$; 

\item $B_4/Z$ does not admit a properly discontinuous action on $X_1$ by semisimple isometries.
\end{enumerate} 
\end{theorem}

We give a construction of the space $X_1$ in \ts\ref{S - X}. To prove the first assertion, we  need to find a group $G_1$ which is of finite index in $G_0=B_4/Z$  and which acts geometrically on $X_1$. This group is described in  \ts\ref{S - G}. It  can be chosen to be the image in the central quotient $B_4/Z$ of the stabilizer of a point under the standard permutation representation of $B_4$.

The standard action of $G_0$ on the Brady complex $X_0$ was studied by Crisp and Paoluzzi in  \cite{crisp2005classification}. In this paper they prove in particular a fundamental rigidity theorem, which shows (see \cite[Theorem 1]{crisp2005classification}) that the Brady complex $X_0$ is, essentially, the only 2-dimensional CAT(0) space on which $G_0$ can act geometrically.  

This statement needs to be formulated carefully, since the standard action of $G_0$ on $X_0$ can always be deformed into a geometric action which is not the standard action, showing that  $X_0$ is not uniquely determined by the group $G_0$ (see Theorem 2 in \cite{crisp2005classification} for a precise statement).  The  rigidity theorem asserts that for every geometric action  $G_0\acts X$ on a CAT(0) space of covering dimension 2, there exists a uniquely determined equivariant mapping $f\colon X_0\to X$ which is locally injective, and locally isometric (up to a constant scaling of the metrics on either $X_0$ or $X$) on the complement of the 0-skeleton of $X_0$; in this sense, the group $G_0$ ``remembers'' both the CAT(0) space $X_0$ and the standard action.

The space $X_1$  is not one of the Crisp--Paoluzzi deformations of $X_0$. More precisely, the following holds.

\begin{theorem}\label{T - th1 prime}
There does not exist a $G_1$-equivariant map $f\colon X_0\to X_1$  which is locally injective and locally isometric (up to a constant scaling of the metrics) on the complement of the 0-skeleton of $X_0$. 
\end{theorem}  

(In this statement the group $G_1$  acts geometrically on $X_0$ through the standard action of $G_0$.) 

\begin{remark} The map $f\colon X_0\to X$ in the Crisp--Paoluzzi rigidity theorem (for the group $G_0$) needs not be  injective. For every $\e>0$, there exists a geometric action $G_0\acts X_\e$ on a 2-dimensional CAT(0) complex $X_\e$ (a deformation of $X_0$), for which the canonical uniquely determined map $f_\e\colon X_0\to X_\e$ provided by the rigidity theorem identifies two distinct orbits or $X_0$ at distance at most $\e$ (see Theorem 2, (ii) in \cite{crisp2005classification}).  
\end{remark}

In the opposite direction, one may wonder if there exists a $G_1$-equivariant mapping $f\colon X_1\to X_0$ subject to similar assumptions, where $G_1$ acts on $X_0$ through the standard action.  The existence of such a map would suggest that the action $G_1\acts X_1$ could play the role of the standard action for the group $G_1$.  
The following implies that such a map also does not exist.

\begin{theorem}\label{T - th2}
There does not exist a CAT(0) space $X$ of covering dimension 2, on which $G_0$ admits a proper  action by semisimple isometries, and a $G_1$-equivariant map $f_1\colon X_1\to X$, where $G_1$ acts geometrically on $X$ as a subgroup of $G_0$, such that $f_1$ is locally injective and locally isometric (up to scaling) outside of the 0-skeleton of $X_1$.  
\end{theorem}

This results show that the Crisp Paoluzzi rigidity theorem does not hold for the group $G_1$, in the sense that there does not exist a geometric action $G_1\acts X_1$, that for every geometric action  $G_1\acts X$ on a CAT(0) space of covering dimension 2, there exists a uniquely determined equivariant mapping $X_1\to X$ which is locally injective, and locally isometric (up to a constant scaling of the metrics on either $X_1$ or $X$) on the complement of the 0-skeleton of $X_1$.   This is rather unusual, since many rigidity statements which hold for a given group (e.g., $G=\SL_3(\ZI)$) also do for their finite index subgroups.  

Finally, we note that (although we are more interested in the group $B_4/Z$ in the present paper)   the above results could also be formulated for the braid group itself. In particular, the following holds.

\begin{theorem}\label{T - virt geom 2}
There exists a CAT(0) complex $Y_1$ of dimension 3 with the following two properties
\begin{enumerate}
\item $B_4$ acts virtually geometrically on $Y_1$; 

\item $B_4$ does not admit a minimal properly discontinuous action on $Y_1$ by semisimple isometries.
\end{enumerate} 
\end{theorem}

As observed in \cite{crisp2005classification} (see \ts3, Prop.\ 5), the classification of geometric actions of $B_4$ on 3-dimensional CAT(0) spaces is essentially equivalent to that of geometric actions of $G_0$ on 2-dimensional CAT(0) spaces.

\toc

\section{Construction of $G_1$}\label{S - G}

We define $G_1$ first.

\subsection{} Consider the presentation 

\begin{align*}
\ip{ a,b,c,d,e,f \mid\ & ba=ae=eb,\ de=ec=cd,\\
& bc = cf = fb,\ df = fa = ad,\\
&ca = ac,\ ef = fe }
\end{align*}

\noindent of $B_4$, where $a,b,c$ are the standard generators, and $e:=aba^{-1}$, $f=cbc^{-1}$,  and $d=(ac)^{-1}bac$ (see  \cite[\ts 3]{crisp2005classification}), and the following two graphs (see \cite[Fig.\ 4]{crisp2005classification}):

\includegraphics[width=8cm]{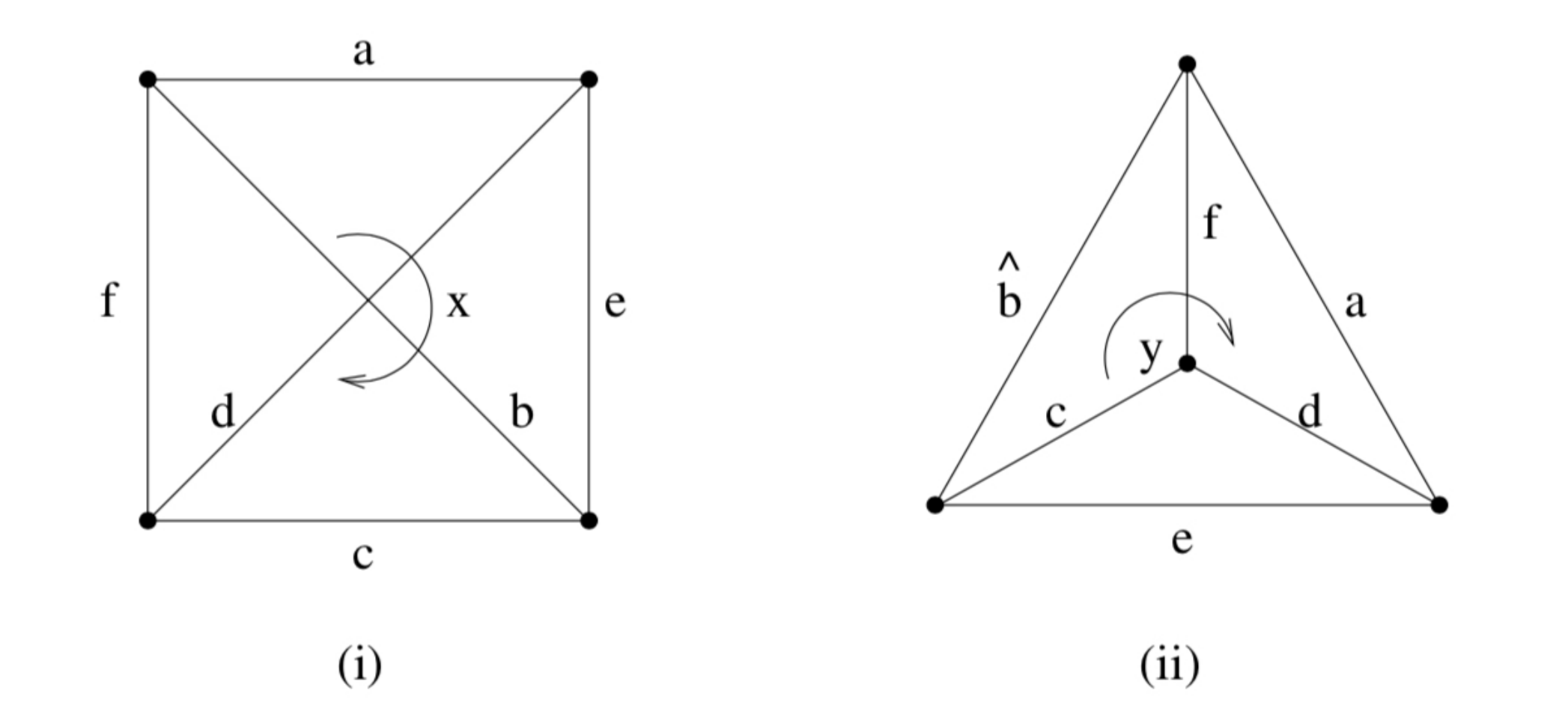}

\noindent where $\hat b:= c^{-1}bc^2$.
The action of $x:=bac$ by conjugation induces a quarter clockwise turn on the first graph,  that of $y:=xc$ a one third clockwise turn on the second  graph. These actions are given by: 
\[
x\colon a\mapsto e\mapsto c\mapsto f\mapsto a,\ b\lra d
\] 
\[
y\colon c\mapsto f\mapsto d\mapsto c,\ a\mapsto e\mapsto \hat b\mapsto a. 
\] 
As in \cite{crisp2005classification}, we shall work with the quotient  group $G_0:=B_4/Z$, where   $Z$ is the center of $B_4$.  It is well--known that the group $Z$ is cyclic generated by $z:=x^4=y^3$.

\subsection{}
We  use the same letters for the elements in $B_4$ and their projections to $G_0$; by the following, we define a subgroup of $G_0$.

\begin{definition}
$G_1:=\ip{a,y}$.
\end{definition}

We first observe that $G_0$ is generated by $x$  and $y$. This follows from the fact that $B_4$ is generated by $a$, $b$, and $c$, and the  equations  
\begin{itemize}
\item[] $c=xy^{-1}$
\item[] $a=xyx^{-2}$ 
\item[] $b=xa^{-1}c^{-1}$.
\end{itemize} The standard relations of $B_4$ translates into the following relations for $G_0$ in terms of the generators $x$ and $y$: 
\[
x^4=y^3=xyx^2y^{-1}x^{-1}y^{-1}x^{-2}y=e.
\]
As mentioned in the introduction, $G_1$ is of finite index in $G_0$. There are several ways to see this. One is to note that the group generated by $a$ and $y$ in $B_4$ maps surjectively onto the stabilizer of a point under the standard permutation representation of $B_4$. Another approach is to map these groups into $\SL_2(\ZI)$. (Note that $G_0$ is a subgroup of index 2 in the automorphism group of $F_2$.)  Namely, it is well known, see \cite[\ts II.1.4]{serre1977arbres}, that $\SL_2(\ZI)$ is generated by two matrices
\[
S=\left (\begin{matrix}0 & 1\\  -1& 0\end{matrix}\right)\text{ and } T=\left (\begin{matrix}1 & 0\\  1& 1\end{matrix}\right)
\]
with the relations $S^4=1$ and $(ST)^3=S^2$. This implies that the map
\[
\begin{cases}
x\mapsto S\\
y\mapsto -ST=S^{-1}T
\end{cases}
\]
extends to a homomorphism $\pi\colon G_0\to SL_2(\ZI)$, since $(-ST)^3=1$, and the element $x^2$ is mapped to the central element $S^2$.

By definition $\pi(G_1)$  is the subgroup of $\SL_2(\ZI)$ generated by $-T$ and $-ST$, which is a subgroup of index 4. Thus, $G_1$ has index 4 in $G_0$. Rouhgly speaking, the group $G_1$ removes the ``$x$-component'' of the torsion from $G_0$. 

\begin{remark}
 $G_0=\ip{a,x}$.
\end{remark}

\subsection{} The Brady complex $X_0$ can be described as follows.  The vertex set of $X_0$ is the set of left cosets of the cyclic group $\ip x$ of order 4 in $G_0$, and the Cayley graph of $G_0$ with respect to $\{a,b,c,d,e,f\}$ is a 4-to-1 simplicial covering of the 1-skeleton of $X_0$. Two distinct types of equilateral triangles are  attached to the 1-skeleton to produce a  space which is CAT(0). We refer to \cite[End of \ts 3]{crisp2005classification} for a description of the attaching maps.
  If three elements $u,v,w$ form a triangle in one of the two graphs drawn above, the group they generate in $B_4$ (which is isomorphic to the braid group on 3 strands) intersects $Z(B_4)$ trivially. We  denote, as in \cite{crisp2005classification}, its image in $G_0$ by $B(u,v,w)$. We also let $t_{u,v}=uv\in B(u,v,w)$ and $z_{u,v}=(uv)^3\in B(u,v,w)$. This group also admits, as Brady and McCammond have shown  \cite{brady2000three}, an action on a CAT(0) space, which   embeds isometrically in $X_0$. Similarly, if three elements $u,v,w$ in $G_1$ form a triangle in one of the graphs drawn above (or a conjugate), they also generate a subgroup $B(u,v,w)$ of $G_1$. 
   
\section{Construction of $X_1$}\label{S - X}

\subsection{}
The next step is to construct a  CAT(0) space $X_1$ on which $G_1$ acts geometrically. Before we do that, we explain the idea behind the definition of $G_1$. The proof of the rigidity theorem in \cite{crisp2005classification}  uses, in particular,   the relative positions of appropriate copies of $B(u,v,w)$ in $G_0$. Starting with a geometric action $G_0\acts X$ on a CAT(0) space of covering dimension 2, a certain region $R$ is defined by the intersection
\[
\Min(f)\cap \Min(z_{f,c})\cap \Min(c)\cap \Min(z_{c,d})\cap \Min(d)\cap \Min(z_{d,f})
\]
This set is convex, compact (possibly empty a priori) and invariant under the action of $y$.    
The proof of \cite[Theorem 2.1]{crisp2005classification}  falls into cases according to the shape of the region $R$.
 Crisp and Paoluzzi prove (see \cite[Prop.\ 8]{crisp2005classification}) that four cases need to be considered
\begin{enumerate}
\item $R$ is empty
\item $R$ is reduced to a single point
\item  $R$ a closed bounded segment which is fixed pointwise by $y$
\item $R$ is two dimensional
\end{enumerate}
They show that none of these cases can in fact occur, except for the last one (4).  

The region $R$ cannot be defined for $G_1$, since the min sets of the elements it involves are not in $G_1$. However, a similar region can be defined, for which arguments similar to that of the Crip--Paoluzzi theorem be developed (to a some extent, since the rigidity theorem ultimately fails for the group $G_1$). 

In particular, the elements $b$ and $e$ both belong to $G_1$, since $y$ conjugates $a$ to $e$, and $b=e^{-1}ae$.  Thus, one can define in analogy with $R$ the region 
\[
\Min(a)\cap \Min(z_{a,e})\cap \Min(e)\cap \Min(z_{e,\hat b})\cap \Min(\hat b)\cap \Min(z_{\hat b,a})
\]
for the group $G_1$ (where $\hat b=c^{-2}bc^2=f^2bf^{-2}$). Again this set is invariant under the action of $y$. We prove that for the group $G_1$,  Case (2) occurs, i.e., this region may be reduced to a single point, by constructing an explicit complex for which this is the case.

\subsection{} To construct the space $X_1$, we use the Brady and McCammond construction  \cite{brady2000three} (see also \cite{crisp2005classification}) of the standard complex for the group $B(u,v,w)$. One of the steps in the proof of the rigidity theorem described in \cite[\ts 4.4]{crisp2005classification} shows, roughly speaking, that one cannot construct a deformation of the standard action of $G_0$ by ``rotating’’ three times the Brady--McCammond construction using the action of $y$. We shall construct the complex $X_1$ precisely in this way.  

Consider in the group  $G_1$ the elements $a$, $b$ and $e$ and the subgroup $B_1:=B(a,e,b)$ they generate inside $G_1$. This group admits  the presentation 
$B_1=\ip{a,e,b\mid ae=eb= ba}$ and
   a geometric action on a  CAT(0) space $X_\theta(a,e,b)$, which is shown in \cite[Fig.\ 2]{crisp2005classification}. In this case, the CAT(0) structure depends on an angle parameter $0<\theta<\pi/2$ for the (pairwise isometric, isosceles) triangles attached to the axis of $z_{a,e}$. In fact, this construction classifies  the minimal 2-dimensional  CAT(0) structures on this group. A proof of this result can be found in \cite[\ts 2]{crisp2005classification}.

  Since $y\in G_1$ acts as an order 3 symmetry, it is not difficult to show that one must choose $\theta = \pi/3$ in the construction of the CAT(0) space $X_1$ defined below; therefore we will  simply denote $Y_1:= X_{\pi/3}(a,e,b)$. It is the universal cover of the locally CAT(0) space $\overline Y_1$ having an equilateral triangle for each of the three expressions of $t_{a,e}$ in the presentation of $B_1$.   We refer to \cite{brady2000three} and \cite{crisp2005classification} for more details on this construction.

Consider in $G_1$ the two subgroups $B_2:=yB_1y^{-1}$ and $B_3:=y^2B_1y^{-2}$. They act respectively on CAT(0) space $Y_2$ and $Y_3$, which are constructed as in the previous two paragraphs, again with the same angle parameter $\theta=\pi/3$.

\begin{definition}\label{D - X1}
We let $X_1$ be the universal complex of the space obtained by identifying $\overline Y_1$, $\overline Y_2$ and $\overline Y_3$ over their common  labeled edges, namely, the loops $e$, $\hat b$ and $a$, respectively in $Y_1\cap Y_2$, $Y_2\cap Y_3$ and $Y_3\cap Y_1$.  
\end{definition}

By definition, $X_1$ is endowed with an action of $y$, such that 
\[
y\colon Y_1\to Y_2\to Y_3\to Y_1,
\] 
and the group $G_1=\ip{a,y}$ acts geometrically on $X_1$.

\begin{lemma}\label{L - X1 cat0}
$X_1$ is a CAT(0) space.
\end{lemma}

\begin{proof}
It suffices to prove the link condition. By construction, the link of a vertex is composed of copies of the links of $B(u,v,w)$ glued together as prescribed by the identifications made in Definition \ref{D - X1}.  The latter graphs are isometric to 
\begin{center}
\includegraphics[width=3cm]{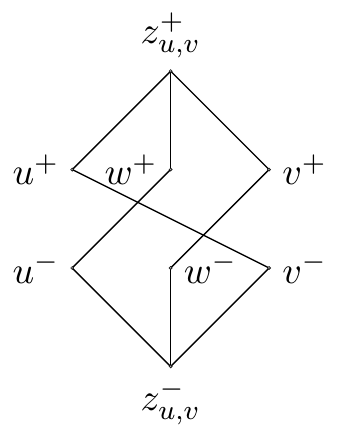} 
\end{center}
In the identification space, the link is  given by 
\begin{center}
\includegraphics[width=10cm]{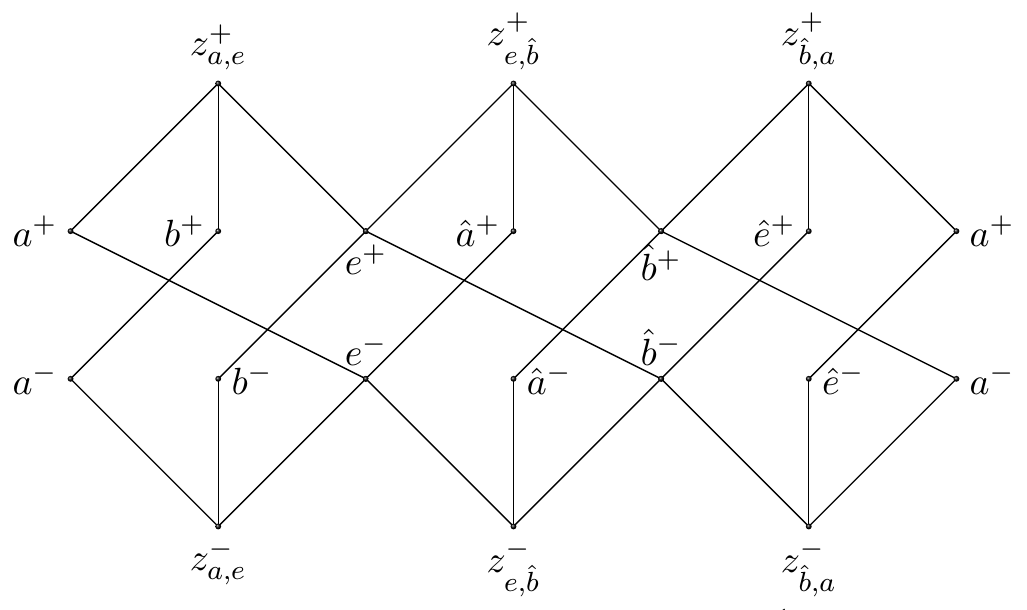} 
\end{center}
It is straightforward to check that the girth of this graph is $2\pi$.
\end{proof}

\begin{remark}
We have therefore obtained the link from \cite[\ts 4.4.\ The case $R\neq\emptyset$ and $\dim(R)\leq 1$, Figure 11]{crisp2005classification}, which is a case that, as Crisp and Paoluzzi have proved, does  not occur for $G_0$. Observe that by construction, the region $R$ defined by
\[
\Min(a)\cap \Min(z_{a,e})\cap \Min(e)\cap \Min(z_{e,\hat b})\cap \Min(\hat b)\cap \Min(z_{\hat b,a})
\]
is reduced to a point.
\end{remark}

\section{Proofs of Theorem \ref{T - virt geom} and Theorem \ref{T - th1 prime}}\label{S - proof theorem 1}

Theorem \ref{T - virt geom} is now a consequence of the following lemma. 

\begin{lemma}\label{L - link embedding}
The link of the Brady complex does not admit a topological, locally isometric embedding into the link of  $X_1$. 
\end{lemma}

\begin{proof}
We recall that the link of the Brady complex is given by

\begin{center}
\includegraphics[width=6cm]{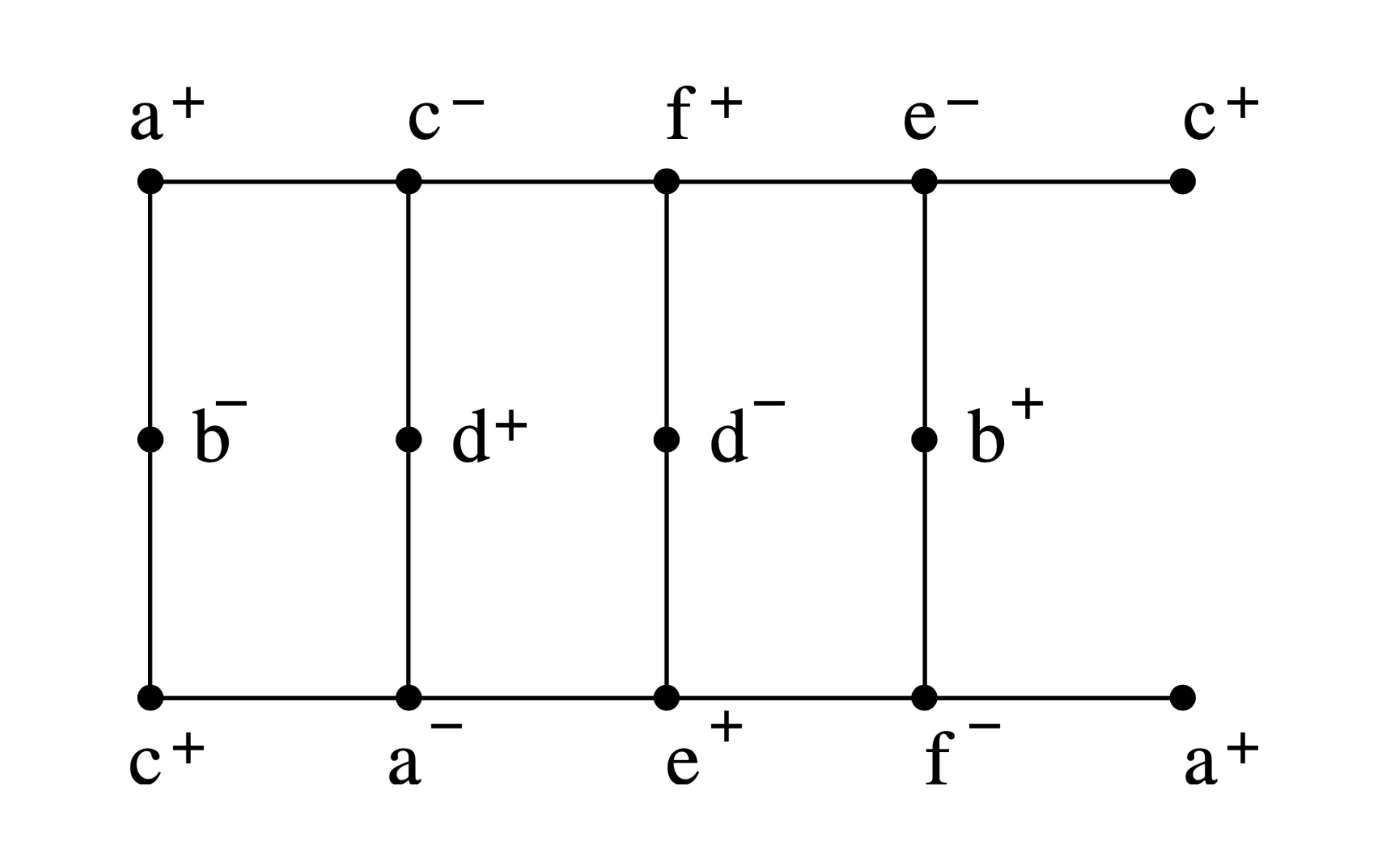}
\end{center}

\noindent (see \cite[Fig. 6]{crisp2005classification}, where the vertex indexing depends on the choice of a representation of the left coset of $\ip x$).

In this graph, consider the Hamiltonian cycle (discarding the vertices of order 2) of length $8\pi/3$ spanning the vertices of order 3. We label the vertices in the graph from 1 to 8 in such a way that there is an edge of length $2\pi/3$ between $k$ and $k+4$ for $k=1,\ldots, 4$. The vertices of order 3 in this graph can be mapped to vertices of order 3 or to vertices of order 4 in the link of $X_1$. By symmetry, we may assume that the vertex 1 goes to $z_{a,e}^+$, and the vertex 4 to $a^{-}$. There are two possible further extensions in the link of $X_1$: the vertex 2 may be mapped to $e^+$ or to $a^+$. In the first case, 5 is mapped to $z_{a,e}^-$, 3 to $z_{e,\hat b}^+$, 6 to $e^{-}$, 4 to $B^+$ and 8 to $z_{e,\hat b}^-$. This is a contradiction, since the distance between $z_{a,e}^+$ and $z_{e,\hat b}^-$ is strictly larger than $\pi/3$. Since the other case also leads to a contradiction in a similar way, the lemma is proved.  
\end{proof}

The following is a direct consequence.

\begin{lemma}
There does not exist a map $f\colon X_0\to X_1$ is locally injective and locally isometric (up to a constant scaling of the metric on $X_1$) on the complement of the 0-skeleton of $X_0$. 
\end{lemma}

The proofs of and Theorem \ref{T - virt geom} and Theorem \ref{T - th1 prime} are similar, and rely on a slightly different set of assumptions, which requires the Crip--Paoluzzi rigidity theorem in the first statement, and assumes it in the second one.  

\begin{proof}[Proof of Theorem \ref{T - virt geom}]
Since $G_1$ acts geometrically on $X_1$ and is of finite index in $G_0$, it remains to prove that $G_0$ does not no semisimple properly discontinuous action on $X_1$ by isometries. This is a straightforward consequence of the previous lemma and the Crip--Paoluzzi theorem.   Namely, 
suppose that $G_0$ admits such an action.  Consider the standard action of $G_0$ on the Brady complex $X_0$. By \cite{crisp2005classification}, Theorem 1, and Remark (2) following the statement of this theorem, there exists an equivariant mapping $f_0\colon X_0\to X_1$ which is locally injective, and locally isometric on the complement of the 1-skeleton of $X_0$.  This contradicts the previous lemma. 
\end{proof}

\begin{proof}[Proof of Theorem \ref{T - th1 prime}]
The assumption is that the Crip--Paoluzzi theorem ``holds for $G_1$ acting on $X_0$’’ in the sense that there exists a $G_1$-equivariant map $f\colon X_0\to X_1$,  with respect to the standard action of $G_0$ on $X_1$, viewing $G_1$ as an embedded subgroup as above, and to the action of $G_1$ on $X_1$ defined above,  which is locally injective and locally isometric (up to a constant scaling of the metric on $X_1$) on the complement of the 0-skeleton of $X_0$. 
Again, this contradicts the previous lemma and shows that $G_1$ admits a geometric action on a space $X_1$ which is not a deformation of  $X_0$. More generally, since the previous lemma holds without an equivariance assumption,  the same is true of every finite index subgroup of $G_1$. 
\end{proof}

\section{Proof of Theorem \ref{T - th2}}

\begin{proof}[Proof of Theorem \ref{T - th2}]
Assume that there exists a proper semisimple action of $G_0$ on a CAT(0) space $X$ of covering dimension 2, such that there exists a $G_1$-equivariant map
\[
f_1\colon X_1\to X
\]
where $G_1$ acts on $X$ through the embedding $G_1\leq G_0$, which is locally injective, and locally isometric (up to scaling)  outside the 0-skeleton of $X_1$.

By the Crisp--Paoluzzi rigidity theorem, there exists a $G_0$-equivariant map 
\[
f_0\colon X_0\to X
\]
which is locally injective, and locally isometric (up to scaling) outside the 0-skeleton of $X_0$.

If $a$ and $b$ are two commuting elements in $G_1$, we denote by $\Pi_0(a,b)$ (resp.\ $\Pi_1(a,b)$, $\Pi(a, b)$) the intersection of $\Min(a)$ and $\Min(b)$ in $X_0$  (resp.\ $X_1$, $X$). In the three cases, these are flat planes, on which the group generated by $a$ and $b$ acts freely with compact quotient. Furthermore, the image of $Y_1$ under $f_1$ is the union of the planes 
\[
f_1(\Pi_1(a, z_{a,e}))=\Pi(a, z_{a,e})
\] 
\[
f_1(\Pi_1(e, z_{a,e}))=\Pi(e, z_{a,e})
\] 
and their translates under the action of $B_1$. A similar fact is true of $Y_i$ for $i=2,3$. Thus, the image of $X_1$ in $X$, being the union of the images of $Y_i$, $i=1,2,3$, is uniquely determined. Similarly, the image $f_0(X_0)$ contains $\Pi(a, z_{a,e})$ and $\Pi(e, z_{a,e})$, and by equivariance, we have $Y_i\subset f_0(X_0)$ for $i=1,2,3$. Therefore, $f_1(X_1)\subset f_0(X_0)$. 

Let $p_0$ be the fixed point of $y$ in $X_0$. The map $f_0$ embeds injectively (and necessarily  isometrically) the link of $p_0$ in $X_0$, which is a circle of length $2\pi$, into that of $p:=f_0(p_0)$ in $X$. By equivariance, since the image $f_0(X_0)$ of the Crip--Paoluzzi map contains a unique fixed point,  the fixed point $p_1$ of $y$ in $X_1$ maps to $p$ under $f_1$. 
Since $f_1(X_1)\subset f_0(X_0)$ and $f_1$ is locally injective, the link of $p_1$ in $X_1$, which is isometric to the graph in Lemma \ref{L - X1 cat0}, maps injectively into the link of $p$ in $f_0(X_0)$. This is absurd, since the latter is a circle.  
\end{proof}

\begin{remark} We have shown that $G_1\acts X_0$ and $G_1\acts X_1$ are two geometric actions of $G_1$, neither of which is can be deformed to the other in the sense intended in the statement of the Crisp--Paoluzzi rigidity theorem. In particular, there does  not exist a ``standard geometric action’’ of the group $G_1$ in dimension 2 which can play the role of the Brady action. Nonetheless, the existence of a well-defined region $R$, analogous to that of Crisp and Poaluzzi, together with the above arguments,  indicates that ``some form of rigidity'' should remain in the group $G_1$ (and perhaps other finite index subgroups). 
\end{remark}

\section{Proof of Theorem \ref{T - virt geom 2}}

Consider the preimage $H_1$ of $G_1$ in $B_4$. This is a finite index subgroup, which is  a central extension of $G_1$. Therefore, there exists a  geometric action of $H_1$ on the metric Cartesian product $Y_1:=\IR\times X_1$, where $H_1$  acts on $\IR$ through the action of $B_4$ induced by the augmentation map $B_4\surj \ZI$. Note that $H_1$ acts freely (since $B_4$ is torsion free, every properly discontinuous action is free).

Suppose that there exists a minimal properly discontinuous action of $B_4$ on $Y_1$ by semisimple isometries. 
 Then by \cite[Theorem II.6.8-15]{bridson1999metric}, $Y_1=\IR\times X_1'$ splits as a metric product, and the central element $z$ acts as a Clifford translation $(t,p)\mapsto (t+|z|,p)$, and the action of $B_4$ on $Y_1$ factorize to a geometric action of $G_0=B_4/Z$ on $X_1'$. By the rigidity theorem, there exists an equivariant map $f\colon X_0\to X_1'$ which is (up to scaling) locally injective and locally isometric on the complement of the zero skeleton of $X_0$. In particular,  $\Lk(X_0)$ admits a topological, locally isometric embedding into the link of $X_1'$ at a vertex. It follows that the suspension $\Sigma \Lk(X_0)$ admits a topological, locally isometric embedding into $\Lk(Y_1)$. Since the latter is isometric to $\Sigma \Lk(X_1)$, there exists a topological, locally isometric embedding of the link of $X_0$ into that of $X_1$. 
  This contradicts Lemma \ref{L - link embedding}.


\begin{thebibliography}{1}

\bibitem{brady1994automatic}
T.~Brady.
\newblock Automatic structures on $\Aut(F_2)$.
\newblock {\em Archiv der Mathematik}, 63(2):97–102, 1994.

\bibitem{brady2000artin}
T.~Brady.
\newblock Artin groups of finite type with three generators.
\newblock {\em The Michigan Mathematical Journal}, 47(2):313–324, 2000.

\bibitem{brady2000three}
T.~Brady and J.~P. McCammond.
\newblock Three-generator artin groups of large type are biautomatic.
\newblock {\em Journal of Pure and Applied Algebra}, 151(1):1–9, 2000.


\bibitem{bridson1999metric}
M.~R. Bridson and A.~Haefliger.
\newblock {\em Metric spaces of non-positive curvature}, volume 319.
\newblock Springer Science \& Business Media, 1999.



\bibitem{crisp2005classification}
J.~Crisp, L.~Paoluzzi.
\newblock On the classification of CAT(0) structures for the 4-string braid
  group.
\newblock {\em The Michigan Mathematical Journal}, 53(1):133–163, 2005.

\bibitem{serre1977arbres}
J.~P. Serre and H.~Bass.
\newblock {\em Arbres, amalgames, $\SL_2$: cours au Coll\`ege de France}.
\newblock Soci\'et\'e math\'ematique de France, 1977.



\end{thebibliography}
\end{document}